\documentclass[a4paper,11pt]{amsart}

\usepackage[a4paper,hmargin={2.5cm,2.5cm},vmargin={2.5cm,2.5cm}]{geometry}
\usepackage{amsthm,amssymb,amsmath}
\usepackage[all,arc]{xy}
\usepackage[mathscr]{euscript}
\usepackage{url}

\newcommand{\pzb}{\subseteq}

\newcommand{\ddn}{\Downarrow}

\newcommand{\tze}{\mid}

\newcommand{\Sup}{\mathcal{S}}
\newcommand{\yoneda}{\textsf{y}}
\newcommand{\unit}{\mathbf{1}}
\newcommand{\dwa}{\mathbf{2}}
\newcommand{\Rw}{\Rightarrow}
\newcommand{\hrw}{\hookrightarrow}
\DeclareMathOperator{\ev}{ev}

\DeclareMathOperator{\ForgetToV}{S}
\DeclareMathOperator{\ForgetToVAd}{A}
\DeclareMathOperator{\MFunctor}{M}

\newcommand{\mate}[1]{\,^\ulcorner\! #1^\urcorner}

\newcommand{\catfont}[1]{\mathbf{#1}}
\newcommand{\V}{\mathcal{Q}}
\newcommand{\two}{\catfont{2}}

\newcommand{\SET}{\catfont{Set}}

\newcommand{\TOP}{\catfont{Top}}

\newcommand{\AP}{\catfont{App}}
\newcommand{\ORD}{\catfont{Ord}}
\newcommand{\MET}{\catfont{Met}}

\newcommand{\Mat}[1]{#1\text{-}\catfont{Rel}}
\newcommand{\Mod}[1]{#1\text{-}\catfont{Dist}}

\newcommand{\Cat}[1]{#1\text{-}\catfont{Cat}}

\newcommand{\Idl}{\catfont{Idl}}
\newcommand{\FC}{\catfont{FC}}
\newcommand{\FSW}{\catfont{FSW}}

\newcommand{\relto}{{\longrightarrow\hspace*{-2.5ex}{\mapstochar}\hspace*{2.3ex}}}
\newcommand{\modto}{{\longrightarrow\hspace*{-2.8ex}{\circ}\hspace*{1.2ex}}}
\newcommand{\kto}{\relbar\joinrel\rightharpoonup}
\newcommand{\krelto}{\,{\kto\hspace*{-2.5ex}{\mapstochar}\hspace*{2.6ex}}}
\newcommand{\kmodto}{\,{\kto\hspace*{-2.8ex}{\circ}\hspace*{1.3ex}}}

\newcommand{\kleisli}{\circ}

\newcommand{\monadfont}[1]{\mathbb{#1}}
\newcommand{\mT}{\monadfont{T}}
\newcommand{\mU}{\monadfont{U}}
\newcommand{\Ultra}{\monadfont{U}}

\newcommand{\Tth}{(\Ultra,\V)}

\newcommand{\op}{\mathrm{op}}

\newcommand{\blackright}{\mbox{ $-\!\mbox{\footnotesize $\bullet$}$ }}
\newcommand{\blackleft}{\mbox{ $\mbox{\footnotesize $\bullet$}\!-$ }}
\newcommand{\whiteright}{\multimap}
\newcommand{\whiteleft}{\mbox{ $\mbox{\footnotesize $\circ$}\!-$ }}

\newcommand{\homkleisliright}{\whiteleft}
\newcommand{\homcompleft}{\blackright}

\theoremstyle{definition}
\newtheorem{definition}{Definition}[section]

\theoremstyle{plain}
\newtheorem{theorem}[definition]{Theorem}
\newtheorem{lemma}[definition]{Lemma}
\newtheorem{corollary}[definition]{Corollary}
\newtheorem{proposition}[definition]{Proposition}

\theoremstyle{remark}
\newtheorem{remark}[definition]{Remark}
\newtheorem{convention}[definition]{Convention}
\newtheorem{example}[definition]{Example}

\begin{document}

\pagestyle{plain}

\begin{abstract}
Our work is a fundamental study of the notion of approximation in $\V$-categories and in $(\Ultra,\V)$-categories, for a quantale $\V$ and the ultrafilter monad $\mU$. We introduce auxiliary, approximating and Scott-continuous distributors, the way-below distributor, and continuity of $\V$- and $(\Ultra,\V)$-categories. We fully characterize continuous $\V$-categories (resp. $(\Ultra,\V)$-categories) among all cocomplete $\V$-categories (resp.\ $(\Ultra,\V)$-categories) in the same ways as continuous domains are characterized among all dcpos. By varying the choice of the quantale $\V$ and the notion of ideals, and by further allowing the ultrafilter monad to act on the quantale, we obtain a flexible theory of continuity that applies to partial orders and to metric and topological spaces. We demonstrate on examples that our theory unifies some major approaches to quantitative domain theory.
\end{abstract}

\keywords{Quantitative domain theory, continuous domain, way-below, Scott-continuity, quantale-enriched category, distributor, complete distributivity}
\subjclass[2010]{06B35, 06D10, 06F07, 18B35, 18D20, 68Q55}

\title{Approximation in quantale-enriched categories}
\author{Dirk Hofmann}
\thanks{The first author acknowledges partial financial assistance by Centro de Investiga\c{c}\~ao e Desenvolvimento em Matem\'atica e Aplica\c{c}\~oes da Universidade de Aveiro/FCT and the project MONDRIAN (under the contract PTDC/EIA-CCO/108302/2008).}
\address{Departamento de Matem\'{a}tica\\ Universidade de Aveiro\\3810-193 Aveiro\\ Portugal}
\email{dirk@ua.pt}
\author{Pawe\l\ Waszkiewicz}
\thanks{The second author is supported  by grant number N206 3761 37 funded by Polish Ministry of Science and Higher Education.}
\address{Theoretical Computer Science, Jagiellonian University, ul.Prof.S.\L ojasiewicza 6, 30-348 Krak\'{o}w, Poland}
\email{Pawel.Waszkiewicz@tcs.uj.edu.pl}

\maketitle

\section{Introduction}

\subsection*{Quantitative domain theory}

The contrast between the needs of denotational semantics and the
modelling power that domain theory can offer became well visible
when in the early Eighties de Bakker and Zucker \cite{deBZ82}
presented a quantitative model of concurrent processes based on
metric spaces. Their work was later further generalized by America
and Rutten \cite{america89} who considered a general problem of
solving recursive domain equations in the category of metric spaces.
Since that time much effort has been spent on unification of
domain-theoretic and metric approaches to denotational semantics,
which in practice meant a search for a class of mathematical
structures that can serve as (quantitative) domains of computation.
As an early example, Smyth proposed a framework based on
quasi-metrics and quasi-uniformities \cite{smyth88}. Both of these
quantitative structures differ from their ``classical'' counterparts
by discarding symmetry. However, in Smyth's opinion, in order to
accommodate semantic domains used in computer science, a further
reformulation of basic definitions involving limits and completeness
was necessary. Consequently, he suggested bicomplete totally bounded
quasi-uniform spaces \cite{smyth91} as quantitative domains, in his
next paper \cite{smyth94} reworked the definition of completeness
(that is named Smyth-completeness since then), and introduced so
called topological quasi-uniform spaces, in which the quasi-uniform
structure is linked to an auxiliary topology by some additional
axioms. Smyth's insight immediately inspired further studies in this
direction \cite{sunderhauf93, kunzi93, sunderhauf95, sunderhauf97}.

Another important structures that unify partial orders and metric
spaces are {\em $\V$-continuity spaces} introduced by Kopperman
\cite{kopperman88}. The idea was to use a non-symmetric distance
that takes values in a set $\V$ with a rich order structure. Flagg
\cite{continuity} suggested that $\V$ should be a {\em value
quantale}, that is, a completely distributive unital quantale in
which the set of elements that are approximated by the unit is a
filter. Soon both authors published a joint paper
\cite{flagg:quantales} summarizing their research.

Since F.W. Lawvere's famous 1973 paper \cite{Law_MetLogClo} it is
well-known that both ordered sets and metric spaces can be viewed as
$\V$-enriched categories of Eilenberg and Kelly \cite{eilenberg66,
kelly}: the former ones for the quantale $\dwa$, the latter ones for
the quantale $[0,\infty]$. Clearly Kopperman and Flagg's
$\V$-continuity spaces are exactly categories enriched in a value
quantale $\V$.

Smyth's and Lawvere's ideas have been  combined together in a series
of papers by the Amsterdam research group at CWI \cite{bonsangue96,
rutten98, rutten:yoneda} that showed, among other things, how to
construct the (sequential Yoneda) completion and powerdomains for
$[0,\infty]$-categories. Their work has been complemented by K\"unzi
and Schellekens in \cite{kunzischellekens} (they proposed the
netwise version of the Yoneda completion). A completion by flat
modules for gmses (resp.\! completion by type 1 filters) was further
discussed by Vickers \cite{vickers} (resp.\! by Schmitt
\cite{schmitt}). Independently, Flagg, S\"underhauf and Wagner
\cite{flagg:essence, FSW} studied ideal completion of
$\V$-continuity spaces and they in effect demonstrated that for
$\V=[0,\infty]$ their results phrased in terms of ideals (called
FSW-ideals here) agree with results of the CWI group phrased in
terms of Cauchy nets. They also gave a representation theory for
{\em algebraic} $\V$-continuity spaces.

Furthermore, in \cite{wagner94} and later in \cite{wagner97}, Wagner
proposed a  framework for solving recursive domain equations in
certain complete $\V$-categories, thereby unifying original attempts
of Scott \cite{Book_ContLat, SP82} and de Bakker and Zucker
\cite{deBZ82} that in the Eighties seemed to be fundamentally
different. Since then these ideas of domain-theoretic origin have
been successfully applied in semantics. Most notably, solving
recursive equations over metric spaces proved to be one of the
fundamental tools in semantics of concurrency, see e.g. \cite{vBr01,
vBW05, vBMOW05, vBHMW07}.

\subsection*{Our motivation and related work}
A central part of domain theory revolves around a notion of {\em
approximation}, which provides a mathematical content to the idea
that infinite objects are given in some coherent way as limits of
their finite approximations. This leads to considering, not
arbitrary complete partial orders, but the continuous ones. Our work
is thought as a fundamental study of the notion of approximation in
$\V$-categories, that generalizes domain-theoretic one. Our
exposition is categorical but kept close to domain-theoretic
language of \cite{abramsky94, Book_ContLat}. Consequently, we speak
about auxiliary, approximating and Scott-continuous
$\V$-distributors, about the way-below $\V$-distributor, and we
introduce continuous $\V$-categories. The generalization from domain
theory to $\V$-categories that we propose proceeds on various
levels, as we shall explain below, comparing our paper to related
work in the area.

\subsubsection*{Relative continuity}
There is no canonical choice for $\V$-categorical counterparts of
even the most fundamental notions of domain theory. For instance, as
we saw above, order ideals can be generalized to several
non-equivalent concepts on the $\V$-level (e.g. forward Cauchy nets,
flat modules, FSW-ideals) which nevertheless yield the same
definitions in both metric and order-theoretic cases \cite{FSW,
rutten:yoneda, vickers}. Consequently, one obtains different notions
of (co)completeness for $\V$-categories based on a specific choice
of ideals. The starting point of our paper is a conviction that one
has not to make this choice right at the beginning, and we study
cocompleteness and continuity of $\V$-categories relative to an
abstract class of ideals $J$ subject to suitable axioms.
Accordingly, we speak about $J$-cocompleteness and $J$-continuity.
As far as there are many papers in the literature dealing with
relative cocompleteness \cite{albert:kelly, kelly:schmitt,
CH_CocomplII, zhang07}, we are not aware of any systematic study of
relative continuity in $\V$-categories. We therefore introduce a
concept of a $J$-continuous $\V$-category and develop its basic
characterisations. For appropriate choices of $\V$ and $J$ we
recover many of the well-known classical structures: continuous
domains, completely distributive complete lattices, Cauchy-complete
metric spaces but there remain many more settings where the meaning
of $J$-continuity is still to be explored.

\subsubsection*{Continuous categories}
The difference between continuous categories of Johnstone and Joyal
\cite{joh82, kos86, ALR03} and our $J$-continuous $\V$-categories is
that the former are $\SET$-based and their continuity is not
relative to the choice of ideals. On the other hand, our
Theorem~\ref{thm:cont}(i) confirms that in essence we introduce
continuity in the same way --- by the requirement that the left
adjoint to the Yoneda embedding itself has a left adjoint.

\subsubsection*{Other relevant literature}
In \cite{Stu_Dynamics} Stubbe considers totally continuous
cocomplete $\V$-categories enriched over a quantaloid $\V$. On the
one hand, a significant part of results from \cite{Stu_Dynamics} can
be recovered from our paper as soon as we fix $J$ to be the class of
all $\V$-distributors. On the other hand, Stubbe shows that instead
with quantales it is possible to work with more general quantaloids.
%It should be acknowledged that our Section~\ref{projcontV} draws ideas from a very interesting discussion of projectivity in \cite{Stu_Dynamics}.

\subsubsection*{$(\mathbb{U},\V)$-categories}
In the last part of our paper we propose a further substantial
generalization of continuous domains by considering so called
$(\mathbb{U},\V)$-categories, where the ultrafilter monad $\Ultra$
is allowed to act on the quantale $\V$. As an example we note that
$(\Ultra,\dwa)$-categories are precisely topological spaces. In
Section \ref{ContTV} we introduce $J$-continuous
$(\Ultra,\V)$-categories and show that defining approximation ---
while still possible `locally' --- becomes difficult globally, which
is of course a price paid for such a generous generality. We close
the paper by giving a full characterization of $J$-continuous
$(\Ultra,\V)$-categories among all $(\Ultra,\V)$-categories in the
same ways as continuous domains are characterized among all dcpos.

It is worth mentioning that the ultrafilter monad $\Ultra$ is made
compatible with the quantale structure $\V$ by the convergence
structure of a compact topology on $\V$. Under some natural
assumptions this topology happens to be the Lawson topology, and
this observation simplifies the presentation of our results.

\section{Preliminaries}
\subsection{Quantales}

A $\V = (Q,\leqslant,\otimes,\unit)$ is a completely distributive
commutative unital quantale (in short: a quantale) such that the
unit element $\unit$ is greatest with respect to the order on $Q$.
We also assume that $\bot\neq \unit$.

Examples of quantales include: the two element lattice $\dwa =
(\{\bot,\unit\},\leqslant, \wedge,\unit)$; the unit interval $[0,1]$
in the order opposite to the natural one, with truncated addition as
tensor; the extended real half line $[0,\infty]$ in the order
opposite to the natural one, with addition as tensor. In general,
every frame with infimum as tensor is a quantale.
%Further examples are mentioned in: \cite{???}.

\subsection{$\V$-categories}
A \emph{$\V$-category} is a set $X$ with a map (called {\em the
structure of} $X$) $X\colon X\times X\to Q$ satisfying
$\unit\leqslant X(x,x)$ (reflexivity), and $X(x,y)\otimes
X(y,z)\leqslant X(x,z)$ (transitivity), for all $x,y,z\in X$. A {\em
$\V$-functor} $f\colon X\to Y$ is a function that satisfies
$X(x,y)\leqslant Y(fx,fy)$ for all $x,y\in X$. The resulting
category $\Cat{\V}$ is isomorphic to the category $\ORD$ of
(pre)ordered sets if $\V=\two$, to the category $\MET$ of
generalized metric spaces \cite{Law_MetLogClo} if $\V=[0,\infty]$ or
$\V=[0,1]$ (metrics bounded by $1$ in the latter case). Furthermore,
$\V$ with its internal hom becomes a $\V$-category. Moreover, any
$\V$-category has its dual $X^{op}$ defined as $X^{op} (x,y) =
X(y,x)$ for all $x,y\in X$.

$\Cat{\V}$ admits a tensor product $X\otimes
Y((x,y),(x',y'))=X(x,x')\otimes Y(y,y')$, and internal hom:
$Y^X(f,g)=\bigwedge_{x\in X}Y(fx,gx)$. The internal hom describes
the pointwise order if $\V=\two$,  and the non-symmetrized
$\sup$-metric if $\V=[0,\infty]$ or $\V=[0,1]$. Since tensor is left
adjoint to internal hom, every $\V$-functor $g\colon X\otimes Y\to
Z$ has its {\em exponential mate} $\mate{g}\colon Y\to Z^X$. For
example, the structure of $X$ is always a $\V$-functor of type
$X^{op} \otimes X\to Q$, and its exponential mate $\yoneda_X\colon
X\to Q^{X^{op}}$, $x\mapsto X(-,x)$ is a $\V$-functor called the
{\em Yoneda embedding}. The Yoneda Lemma then states that for any
$\phi\in \widehat{X}$ (where $\widehat{X} := Q^{X^{op}})$ we have
$\phi x = \widehat{X}(\yoneda_X x,\phi)$.

\subsection{$\V$-distributors}
A $\V$-functor of type $X^{op} \otimes Y\to Q$ is called a {\em
$\V$-distributor}. Examples:
\begin{itemize}
\item[---] The structure of any $\V$-category $X$ is a
$\V$-distributor.

\item[---] Any two $\V$-distributors $\phi\colon X^{op}\otimes Y\to Q$ and
$\psi\colon Y^{op} \otimes Z\to Q$ can be composed to give a
$\V$-distributor of type $X^{op} \otimes Z\to Q$:
$$(\psi \cdot \phi) (x,z) := \bigvee_{y\in Y} (\phi(x,y)\otimes \psi(y,z)).$$
Therefore we think of $\phi\colon X^{op}\otimes Y\to Q$ as an arrow
$\phi\colon X\modto Y$, which, by the above, can be composed with
$\psi\colon Y\modto Z$ to give $\psi\cdot \phi\colon X\modto Z$. Note
also that $Y\cdot \phi = \phi = \phi\cdot X$.

\item[---] Any function $f\colon X\to Y$ gives rise to two
$\V$-distributors, namely $f_*\colon X\modto Y$, $f_*(x,y) =
Y(fx,y)$ and $f^*\colon Y\modto X$, $f^*(y,x) = Y(y,fx)$.
\item[---] For any $\phi\colon X\modto Y$ and $\psi \colon Z\modto
Y$ we define {\em lifting of $\psi$ along $\phi$} to be the
following $\V$-distributor:
$$(\phi\homcompleft \psi)(z,x)=\bigwedge_{y\in Y}Q(\phi(x,y),\psi(z,y)).$$
\end{itemize}

We further observe that for any element $x\colon 1\to X$ ($1$ is the
one-element $\V$-category that should not be confused with the unit
of the quantale), the distributor $x^*\colon X\modto 1$ is in fact
the same as the $\V$-functor $\yoneda_X x := X(-,x)\in \widehat{X}$.

In $\ORD$, distributors of type $X\modto 1$ are precisely
(characteristic maps of) lower sets, and distributors of type
$1\modto X$ are upper sets of the poset $X$.

On the other hand, in $\MET$, any Cauchy sequence
$(x_n)_{n\in\omega}$ induces a distributor $\phi\colon 1\modto X$
via $\phi(x) = \lim_{n\to\infty} X(x_n,x)$, and a distributor
$\psi\colon X\modto 1$ via $\psi(x) = \lim_{n\to\infty} X(x,x_n)$.
Observe that $\psi\cdot \phi \leqslant 0$ and $\phi\cdot
\psi\geqslant X$ in the pointwise order. Conversely, any pair of
distributors that satisfies the above equations comes from some
Cauchy sequence on $X$.

More generally, we will say that $\V$-distributors $\phi\colon
Z\modto X$, $\psi \colon X\modto Z$ are adjoint iff $\phi\cdot
\psi\leqslant X$ and $\psi\cdot\phi\geqslant Z$. In this case we say
that $\phi$ is a left adjoint to $\psi$ and $\psi$ is a right
adjoint to $\phi$.

\section{$J$-cocomplete $\V$-categories}

Suppose for each $\V$-category $X$ there is given a collection $JX$
of $\V$-distributors of type $X\modto 1$ (called thereafter {\em $J$-ideals}) such that $JX$ contains $x^*\in JX$, for every $x\in X$. This condition on $JX$ tells us in fact that the Yoneda
embedding $\yoneda_X\colon X\to \widehat{X}$ corestricts to $JX$. We
now define $X$ to be {\em $J$-cocomplete} if $\yoneda_X\colon X\to
JX$ has a left adjoint in $\Cat{\V}$. That is, there must exist a
$\V$-functor $\Sup_X\colon JX\to X$ such that for all $\phi\in JX$
and all $x\in X$:
\begin{equation}
\label{eq:ideal}
X(\Sup_X\phi, x) = \widehat{X}(\phi,\yoneda_X x).
\end{equation}
The element $\Sup_X\phi\in X$ is called the {\em supremum} of
$\phi$.

If $JX = \widehat{X}$, then evidently $\widehat{X}$ itself is
cocomplete (meaning: $\widehat{X}$-cocomplete), since the supremum
of $\psi\colon \widehat{X}\modto 1$ is $\psi\cdot \yoneda_X^*$. For
example, if $\V=\dwa$, then $\widehat{X}$ is a poset of lower
subsets of the poset $X$ ordered by inclusion, $\psi$ is a lower set
of lower sets of $X$, and the supremum of $\psi$ is nothing else but
$\bigcup\psi$.

Unfortunately, in general $JX$ is not itself $J$-cocomplete, since
(in the example above) even if $\phi\in JX$ and $\yoneda^*\in JX$,
their composition may not be an element of $JX$. On the other hand,
closure under composition provides exactly what is needed:

\begin{theorem}
The following are equivalent:
\begin{enumerate}
\item for all $\psi\in JJX$, $\psi\cdot \yoneda^*\in JX$,
\item the inclusion $JX\hookrightarrow \widehat{X}$ is closed under
$J$-suprema,
\item $\Mod{J} := \{\phi\colon X\modto Y\tze \forall y\in Y\ \ y^*\cdot
\phi\in JX\}$ is closed under composition.
\end{enumerate}
\end{theorem}

To summarize if $\Mod{J}$ is closed under 'upper stars' and
composition, then $JX$ is $J$-cocomplete, for every $X$. In this case
$\Mod{J}$ is called a {\em saturated class of weights}.

\begin{convention}
In the rest of our paper we consider only saturated classes of
weights.
\end{convention}

\noindent Relative cocompleteness allows for a unified presentation
of seemingly unrelated notions of order- and metric completeness:

\begin{example}
\label{ex:ord} For $\V=\two$, we consider all $\V$-distributors of
type $X\modto 1$ corresponding to order-ideals in $X$ (i.e. directed
and lower subsets of $X$), and write $J= \Idl$. Then $X$ is
$\Idl$-cocomplete iff $X$ is a directed-complete.
\end{example}
\begin{example}
\label{ex:met} For $\V=[0,\infty]$ we consider all $\V$-distributors
of type $X\modto 1$ corresponding to ideals in $X$ in the sense of
\cite{rutten:yoneda}, and write $J=\FC$. These ideals in turn
correspond to equivalence classes of forward Cauchy sequences on
$X$. Hence, $X$ is $\FC$-cocomplete if and only if each forward
Cauchy sequence on $X$ converges iff $X$ is sequentially Yoneda
complete.
\end{example}

\begin{example}
For any $\V$ we can choose $J$ to consist of all right adjoint
$\V$-distributors (i.e. $\V$-distributors that have left adjoints).
Recall from \cite{Law_MetLogClo} that, for $\V = [0,\infty]$ and for
$\V=[0,1]$, a right adjoint $\V$-distributor $X\modto 1$ corresponds
to an equivalence class of Cauchy sequences on $X$. A generalized
metric space $X$ is $J$-cocomplete iff each Cauchy sequence on $X$
converges.
\end{example}

\begin{example}
\label{ex:FSWideals} For a completely distributive quantale $\V$ and
any $\V$-category $X$, a $\V$-distributor $\psi\colon X\modto 1$ is
a $\FSW$-ideal if: (a) $\bigvee_{z\in X} \psi z = \unit$, and (b)
for all $e_1,e_2,d\prec \unit$, for all $x_1,x_2\in X$, whenever
$e_1\prec \psi x_1$ and $e_2\prec \psi x_2$, then there exists $z\in
X$ such that $d\prec \psi z$, $e_1\prec X(x_1,z)$ and $e_2\prec
X(x_2,z)$. Now for $\V = [0,\infty]$ $\FSW$-ideals on $X$ are in a
bijective correspondence with equivalence classes of forward Cauchy
nets on $X$ \cite{FSW}; for $\V=\dwa$, $\FSW$-ideals are
characteristic maps of order-ideals on $X$. Therefore this example
unifies Examples \ref{ex:ord}, \ref{ex:met}.
\end{example}

\noindent Further examples are mentioned in \cite{schmitt,CH_CocomplII, zhang}.

\section{$J$-continuous $\V$-categories}
\label{section:JcocoVcat}

We now come to the main subject of this paper and introduce $J$-continuous $\V$-categories that provide generalization for many structures that play a major role in theoretical computer science, e.g. continuous domains, complete metric spaces, or completely distributive complete lattices.

Let $J_SX$ be a subset of $JX$ consisting of these $J$-ideals that
have suprema, i.e. $\phi\in J_SX$ iff there exists $\Sup_X\phi\in X$
such that the equation (\ref{eq:ideal}) holds for all $x\in X$.
Observe that this enables us to consider, for any $\V$-category $X$,
the $\V$-functor $\Sup_X\colon J_SX\to X$. Moreover, we note that by
the  Yoneda lemma, for any $x\in X$, $\Sup\yoneda_X x = x$ and hence
$\yoneda_X\colon X\to JX$ further corestricts to $\yoneda_X\colon
X\to J_SX$.

\begin{convention}
In what follows we will drop the indices in $\Sup_X$ and $\yoneda_X$
if the context allows us to do so. \end{convention}

\begin{definition}
A $\V$-category $X$ is \emph{$J$-continuous} if the supremum
$\Sup\colon J_SX\to X$ has a left adjoint.
\end{definition}

Note that any $\V$-functor of type $X\to J_SX$ corresponds to a
certain $\V$-distributor $X\modto X$ belonging to $J$. Hence, $X$ is
$J$-continuous if and only if there exists a $\V$-distributor
$\Downarrow:X\modto X$ necessarily in $J$ so that, moreover,
$\mate{\Downarrow}$ is of type $X\to J_SX$ and is left adjoint to
$\Sup\colon J_SX\to X$.

Let us locate $\Downarrow$ among other $\V$-distributors of the same
type. Firstly, for any $\V$-distributor $v\colon X\modto X$ one has:
\begin{align*}
\forall\psi\in J_SX\ \ (\mate{v}\cdot\Sup(\psi)\leqslant\psi)
&\ \ \mathrm{iff} \ \ \forall\psi\in J_SX\ \forall x\in X\ \ (v(x,\Sup\psi)\leqslant\psi x)\\
&\ \ \mathrm{iff} \ \ \forall\psi\in J_SX\ \forall x\in X\ \ ((\Sup^*\cdot v)(x,\psi)\leqslant\yoneda_*(x,\psi))\\
&\ \ \mathrm{iff} \ \ \Sup^*\cdot v\leqslant\yoneda_*.
\end{align*}
In particular, $\Sup^*\cdot\Downarrow\leqslant\yoneda_*$, and
$\Downarrow:X\modto X$ is the largest such $\V$-distributor
since, for every $\V$-distributor $v:X\modto X$,
\begin{align*}
\Sup^*\cdot v\leqslant\yoneda_* &\ \ \mathrm{implies} \ \ \forall x\in
X\ \ ((\mate{v}\cdot\Sup)(\mate{\Downarrow}x)\leqslant
\mate{\Downarrow}x)\\
&\ \ \mathrm{implies} \ \ \forall x\in
X\ \ \mate{v}x\leqslant \mate{\Downarrow}x\\
&\ \ \mathrm{implies} \ \  v\leqslant\ \Downarrow.
\end{align*}

We have identified $\Downarrow:X\modto X$ as the lifting $\ddn =
\Sup^* \blackright \yoneda_*$ of $\yoneda_*:X\modto J_SX$ along
$\Sup^*\colon X\modto J_SX$. Of course, this lifting exists in any
$\V$-category and can be studied in its own right. In
Section~\ref{way-below-dist} we will do so, and give conditions which guarantee
that it provides a left adjoint to $\Sup\colon J_SX\to X$.

Turning to the classical case $\V=\two$ and $J=\Idl$, the distributor $\Downarrow$ is given by the way-below relation. Therefore, we will call the distributor $\Downarrow\colon X\modto X$ the {\em way-below $\V$-distributor} on $X$. In the case of metric spaces, as a consequence of symmetry, $\Downarrow\colon X\modto X$ is the same as the structure $X\colon X\modto X$.

As it is well-known, the way-below relation on a continuous dcpo is the smallest approximating auxiliary relation.  In what follows, we aim for a similar characterisation of the way-below \mbox{$\V$-distributor} in the general case.

\begin{definition}
Define a $\V$-distributor $v\colon X\modto X$ to be:
\begin{enumerate}
\item[---] {\em auxiliary}, if $v\leqslant X$.
\item[---] {\em interpolative}, if $v\leqslant v\cdot v$;
\item[---] {\em approximating}, if $v\in J$ and $X\blackleft v = X$.
\end{enumerate}
Furthermore, a $\V$-distributor $v\colon X\modto Y$ is:
\begin{enumerate}
\item[---] {\em $J$-cocontinuous}, if $\Sup^*\cdot v = \yoneda_*\cdot v$.
\end{enumerate}
\end{definition}
For example $\Idl$-continuous $\dwa$-distributors of type $X\modto 1$ are precisely the (characteristic maps) of Scott-open subsets of $X$. 

\subsection{Approximating $\V$-distributors}

Approximating $\V$-distributors naturally generalize approximating relations on posets, and enjoy analogous properties:
\begin{lemma}\label{lem:approxV}
Every approximating $\V$-distributor $v\colon X\modto X$ is auxiliary. If $v, w\colon X\modto X$ are approximating, then so is $w\cdot v$.
\end{lemma}
\begin{proof}
If $v$ is approximating, then $v=X\cdot v=(X\blackleft v)\cdot
v\leqslant X$. Let now  $v, w\colon X\modto X$ be approximating
$\V$-distributors. By hypothesis on $J$, $w\cdot v\in J$.
Furthermore, $X\blackleft (w\cdot v) = (X\blackleft v)\blackleft w = X$.
\end{proof}

\begin{lemma}\label{charapproxV}
A $\V$-distributor $v\colon X\modto X$ is approximating if and only
if its exponential mate $\mate{v}$ is of type $X\to J_SX$ and
$\Sup\mate{v} = 1_X$.
\end{lemma}
\begin{proof}
By definition, $v\colon X\modto X$ is approximating if and only if
$\mate{v}$ is of type $X\to JX$ and, for each $x\in X$,
$x_*=X\blackleft (x^*\cdot v)$. This in turn is equivalent to
$\mate{v}x\in J_SX$ and $(\Sup\cdot\mate{v})(x)=x$, for each $x\in
X$.
\end{proof}

\begin{lemma}\label{lem:approxScottV}
Any approximating $J$-cocontinuous $\V$-distributor is interpolative.
\end{lemma}
\begin{proof} From $\Sup^*\cdot v = \yoneda_*\cdot v$ we deduce
$v= {\mate{v}}^* \cdot \Sup^*\cdot v={\mate{v}}^* \cdot\yoneda_*\cdot v=v\cdot v$.
\end{proof}

\subsection{$J$-cocontinuous $\V$-distributors}
\label{sect:SCdistributors}

\begin{definition}
\label{def:distCont} A $\V$-functor $f\colon X\to Y$ is
\emph{$J$-cocontinuous} if $f(\Sup\varphi) = \Sup
\underline{f}(\varphi)$, for all $\varphi\in JX$, providing that
both suprema exist. Here $\underline{f}(\varphi)$ is defined to be
$\varphi\cdot f^*$.
\end{definition}

\noindent
As canonical examples we consider the following:
\begin{itemize}
\item[---] $\Idl$-cocontinuous $\dwa$-functors are precisely $J$-cocontinuous maps between posets.
\item[---] $\FSW$-cocontinuous $\dwa$-functors are precisely $J$-cocontinuous maps between posets.
\item[---] $\FC$-cocontinuous $[0,\infty]$-functors are precisely the maps that preserve limits of forward
Cauchy sequences.
\item[---] $\FSW$-cocontinuous $[0,\infty]$-functors are precisely the maps that preserve limits of forward
Cauchy nets.
\end{itemize}

\begin{proposition}\label{ScottDistFunV}
A $v\colon X\modto Y$ is $J$-cocontinuous iff $\mate{v}\colon Y\to\widehat{X}$ is
$J$-cocontinuous.
\end{proposition}
\begin{proof}
It is routine to check that for any $\psi\in J_SY$, $\Sup_X(\underline{\mate{v}}(\psi)) = \psi\cdot v$. Hence the $\V$-functor $\mate{v}$ is $J$-cocontinuous iff $\mate{v}(\Sup_Y\psi) = \psi\cdot v$. We have, for $x\in X$: $\mate{v}(\Sup_Y\psi)(x) = v(x,\Sup_Y\psi) = (\Sup_Y^*\cdot v)(x,\psi) = ((\yoneda_Y)_* \cdot v)(x,\psi) = (\mate{\psi}{}^* \cdot (\yoneda_X)_*\cdot v)(x) = (\psi\cdot v)(x)$, as required.
\end{proof}

\begin{corollary}
If $v\colon Y\modto Z$ is $J$-cocontinuous, then $v\cdot w\colon X\modto Z$ is \mbox{$J$-cocontinuous,} for any $w\colon X\modto Y$.
\end{corollary}
\begin{corollary}
A $\V$-distributor $v\colon X\modto Y$ is $J$-cocontinuous if and only if $v\cdot x_*\colon X\modto Y$ is \mbox{$J$-cocontinuous} for all $x\in X$.
\proof From $\Sup^*\cdot v\cdot x_*=\yoneda_*\cdot v\cdot x_*$ for all $x\in X$ we deduce $\Sup^*\cdot v=\yoneda_*\cdot v$.\qed
\end{corollary}

\subsection{The way-below $\V$-distributor}\label{way-below-dist}
Recall from the beginning of Section~\ref{section:JcocoVcat}, that we define the {\em way-below} $\V$-distributor $\ddn\colon X\modto X$ to be the largest $v$ such that $\Sup^*\cdot v\leqslant \yoneda_*$, that is, $\ddn := \Sup^* \blackright \yoneda_*$.
$$\xymatrix{J_SX & \ar@{->}|-{\object@{o}}[l]_{\yoneda_*}\ar@{.>}|-{\object@{o}}[ld]^\ddn X\\ X\ar@{->}|-{\object@{o}}[u]^{\Sup^*}}$$

As in the poset case, the way-below $\V$-distributor is not, in general, approximating; however, it is smaller than any approximating $\V$-distributor:
\begin{lemma}
\label{lemma:AppImpliesLeqWaybelow}
If $v\colon X\modto X$ is approximating, then $\ddn \ \leqslant v$. Hence, the way-below $\V$-distributor is auxiliary.
\proof Since $\ulcorner v\urcorner^*\cdot\yoneda_* \leqslant v$, we have $\yoneda_* \leqslant \ulcorner v\urcorner^*\blackright v$. Hence $\ddn \ = \Sup^* \blackright \yoneda_* \leqslant \Sup^* \blackright (\ulcorner v\urcorner^*\blackright v) = \ulcorner v\urcorner^*\cdot \Sup^* \blackright v = X \blackright v = v$.\qed
\end{lemma}
\begin{corollary}
If $\ddn$ is approximating, then $\ddn$ is interpolative.
\proof
If $\ddn$ is approximating, then so is $\ddn\cdot\ddn$, and therefore $\ddn\ \leqslant\ \ddn\cdot\ddn$.\qed
\end{corollary}

\begin{lemma}
\label{lemma:AuxSCImpliesLeqWaybelow}
Any auxiliary $J$-cocontinuous $v\colon X\modto X$ satisfies $v\leqslant \ \ddn$.
\proof $\Sup^*\cdot v \leqslant \yoneda_*\cdot v \leqslant \yoneda_*\cdot X = \yoneda_*$. Therefore $v\leqslant\Sup^*\blackright \yoneda_* = \ \ddn$.\qed
\end{lemma}
\begin{lemma}
\label{lemma:IntLeqWaybelowImpliesScott}
If $v\colon X\modto X$ is interpolative and $v\leqslant \ddn$, then $v$ is $J$-cocontinuous.
\proof $v\leqslant\Sup^*\blackright \yoneda_*$ if and only if $\Sup^*\cdot v \leqslant \yoneda_*$, which yields $\Sup^*\cdot v \leqslant \Sup^*\cdot v \cdot v\leqslant \yoneda_*\cdot v$.\qed
\end{lemma}
Also from \cite{Stu_Dynamics} we have:
\begin{lemma}
\label{lemma:SplitIsAdjoint} Let $\alpha\colon X\to J_SX$ be a
$J$-cocontinuous $\V$-functor with $\Sup\alpha \cong 1$. Then
$\alpha \dashv \Sup$. 
\end{lemma}

We gather the most important consequences of the above considerations here:
\begin{theorem}
\label{thm:cont}
Let $X$ be a $\V$-category and let $v\colon X\modto X\in J$. The following are equivalent:
\begin{enumerate}
\renewcommand{\theenumi}{\roman{enumi}}
\item $\mate{v}$ is of type $X\to J_SX$ and $\mate{v}\dashv\Sup$,
\item $v$ is approximating and $v=\ \ddn$,
\item $v$ is approximating and $J$-cocontinuous,
\item $v$ is approximating and $\mate{v}\colon X\to J_SX$ is $J$-cocontinuous,
\item for all $x\in X$ and $\varphi\in J_SX$ we have $\widehat{X}(\mate{v}(x),\varphi) = X(x,\Sup\varphi)$.
\end{enumerate}
\proof The implication (i)$\Rw$(ii) we have already discussed at the
beginning of this section. To see (ii)$\Rw$(iii), assume that $\ddn$
is approximating. Then $\ddn$ is interpolative and therefore
$J$-cocontinuous. Assume now (iii). Then $\mate{v}\colon
X\to\widehat{X}$ is $J$-cocontinuous. Therefore also $\mate{v}\colon
X\to J_SX$, which shows that (iii)$\Rw$(iv). Lemma
\ref{charapproxV} and Lemma \ref{lemma:SplitIsAdjoint} imply
immediately (iv)$\Rw$(i). Clearly, (i) implies (v). Finally, assume
(v). Then
\[
(X\blackleft v)(x,y) = \widehat{X}(\mate{v}(x),\mate{X}(y)) =
\widehat{X}(\mate{v}(x),\yoneda(y)) = X(x,\Sup \yoneda(y)) = X(x,y),
\]
which proves that $v$ is approximating. Hence, $\mate{v}$ is of type
$X\to J_SX$ and indeed left adjoint to $\Sup$. \qed
\end{theorem}

The following theorem provides a full characterization of $J$-continuity of $\V$-categories:

\begin{theorem}
The following are equivalent, for a $\V$-category $X$.
\begin{enumerate}
\renewcommand{\theenumi}{\roman{enumi}}
\item $X$ is $J$-continuous,
\item The way-below $\V$-distributor $\ddn \colon X\modto X$ is approximating,
\item There exists a $J$-cocontinuous approximating $\V$-distributor $v \colon X\modto X$.
\end{enumerate}
\end{theorem}

\noindent
Examples:
\begin{enumerate}
\item[---] $\FSW$-continuous $\dwa$-categories are precisely continuous domains;
\item[---] cocontinuous $\dwa$-categories are completely distributive complete lattices; the way-below distributor becomes the `totally-below' relation associated with complete distributivity of the underlying lattice;
\item[---] $[0,\infty]$ considered with the generalized metric structure $[0,\infty](x,y) = \max\{y-x,0\}$ is an $\FSW$-continuous $[0,\infty]$-category;
\item[---] complete metric spaces are $\FSW$-continuous $[0,\infty]$-categories.
\end{enumerate}

\section{$J$-continuous $(\Ultra,\V)$-categories}

Besides metric spaces, also other geometric objects such as topological and approach spaces can be viewed as generalized ordered sets. The topological case is very elegantly expressed in \cite{Bar_RelAlg} where topological spaces are presented as sets $X$ equipped with a relation $\mathfrak{x}\to x$ between ultrafilters and points, subject to the \emph{reflexivity} and the \emph{transitivity} condition
\begin{align}\label{TopAx}
\dot{x}\to x, &&& (\Upsilon\to\sigma\;\&\;\sigma\to x)\,\Rightarrow\, m_X(\Upsilon)\to x,
\end{align}
for all $x\in X$, $\sigma\in UX$ and $\Upsilon\in UUX$. Here $e_X(x)=\dot{x}$ is the principal ultrafilter induced by $x$ and
\begin{align*}
m_X(\Upsilon)=\{A\subseteq X\mid A^\#\in\Upsilon\} && (A^\# = \{\sigma\in UX\mid A\in\sigma\})
\end{align*}
is the filtered sum of the filters in $\Upsilon$. Furthermore, approach spaces \cite{Low_ApBook} are to topological spaces what metric spaces are to ordered sets: one trades the quantale $\dwa$ for $[0,\infty]$. Hence, an approach space can be presented as a pair $(X,a)$ consisting of a set $X$ and a $[0,\infty]$-relation $a:UX\relto X$ satisfying
\begin{align}\label{AppAx}
0\ge a(\dot{x},x) &&\text{and}&& Ua(\Upsilon,\sigma)+a(\sigma,x)\ge a(m_X(\Upsilon),x),
\end{align}
and a mapping $f:X\to Y$ between approach spaces $X=(X,a)$ and $Y=(Y,b)$ is a \emph{contraction} whenever $a(\sigma,x)\ge b(Uf(\sigma),f(x))$ for all $\sigma\in UX$ and $x\in X$. In the sequel $\AP$ denotes the category of approach spaces and contraction maps. It is now a little step to admit that the domain $\mathfrak{x}$ of $\mathfrak{x}\to x$ in $X$ is an element of a set $TX$ other then the set $UX$ of all ultrafilters of $X$. Eventually, we reach at the notion of a  $(\mT,\V)$-category, for a $\mathsf{Set}$-monad $\mT=(T,e,m)$ and quantale $\V$, as introduced in \cite{CH_TopFeat,CT_MultiCat,CHT_OneSetting}. However, to keep our presentation simple, in this paper we decided to limit our choice of monad to $\Ultra$ (the identity monad case already implicitly discussed in preceeding sections) but we hasten to remark that the majority of the results that follow can be restated and proved in the general setting.

\subsection{The ultrafilter monad}

The ultrafilter monad $\Ultra = (U,e,m)$ consists of:
\begin{itemize}
\item[---] a functor $U\colon \SET\to\SET$ that to each set $X$ assigns the set of all ultrafilters on $X$, and to each map $f\colon X\to Y$ assigns a map $Uf\colon UX\to UY$ given by $Uf(\sigma) := \{B\pzb Y\tze f^{-1}[B]\in \sigma\}$;
\item[---] the unit $e$, which is a natural transformation from the identity functor on $\SET$ to $U$ given componentwise by: $e_X\colon X\to U X$, $e_X(x) := \dot{x} = \{A\pzb X\tze x\in A\}$;
\item[---] the multiplication $m$, which is a natural transformation of type $UU\to U$. Its component $m_X\colon UU X\to UX$ assigns to each ultrafilter of ultrafilters $\Upsilon$ these subsets $A$ of $X$ for which $A^\# = \{\sigma\in UX\mid A\in\sigma\}$ belongs to $\Upsilon$. %Formally: $m_X(\Upsilon) := \{A\pzb X\tze \dot{A}\in \Upsilon \}$, where $\dot{A} = \{\mathcal{A}\in \power\power X\tze A\in \mathcal{A}\}$.
\end{itemize}

\subsection{The Lawson topology on $\V$}

Note that the transitivity axiom in both \eqref{TopAx} and \eqref{AppAx} above involves the application of $U$ to a relation $r:X\relto Y$: for a $\dwa$-relation one puts
\[
\sigma\,(Ur)\,\nu \;\;\; \text{ if }\;\;\; \forall A\in\sigma,B\in\nu\,\exists x\in A,y\in B\,.\,x\,r\, y,
\]
and for a $[0,\infty]$-relation
\[
Ur(\sigma,\nu)=\sup_{A\in\sigma,B\in\nu}\inf_{x\in A,y\in B}r(x,y),
\]
where $\sigma\in UX$ and $\nu\in UY$. These examples suggest that, for a general quantale $\V$, one defines
\[
Ur(\sigma,\nu)=\bigwedge_{A\in\sigma,B\in\nu}\bigvee_{x\in A,y\in B}r(x,y).
\]
in order to obtain an extension of $U$ to $\V$-relations. It is interesting to observe that this formula can be rewriten as 
\begin{align*}
(\sigma,\nu) &\mapsto\bigvee \{\xi\cdot Ur(\omega)\;\Bigl\lvert\;\omega\in U(X\times Y), U\pi_1(\omega)=\sigma,U\pi_2(\omega)=\nu \}
\end{align*}
where $\xi:U\V\to\V, \sigma\mapsto \bigwedge_{A\in \sigma}\bigvee A$ is the (convergence of the) Lawson topology on $\V$ \cite[Thm.III-3.17]{Book_ContLat}. Being the convergence of a compact Hausdorff topology on $\V$, the diagrams
\begin{align*}
\xymatrix{Q\ar[r]^{e_Q}\ar[dr]_{1_Q} & UQ\ar[d]^\xi\\ & Q} &&& \hspace{3em}
\xymatrix{UUQ\ar[d]_{m_Q}\ar[r]^{U\xi} & UQ\ar[d]^\xi\\ UQ\ar[r]_\xi & Q}\\
\end{align*}
commute. Furthermore, in order to guarantee functoriality and other good properties of the above extension of $U$ to $\Mat{\V}$ we assume that
\begin{itemize}
\item[---] $(\V,\unit,\otimes)$ is a monoid in the category of compact Hausdorff spaces, i.e.\ the diagrams
\begin{align*}
\xymatrix{U1\ar[d]_{!}\ar[r]^{U\unit} & UQ\ar[d]^\xi\\ 1\ar[r]_{\unit} & Q} &&&
\xymatrix{U(Q\times Q)\ar[rr]^-{U(\otimes)}\ar[d]_{\langle\xi\cdot U\pi_1,\xi\cdot U\pi_2\rangle} && UQ\ar[d]^\xi\\ Q\times Q\ar[rr]_-{\otimes} && Q}
\end{align*}
commute, and 
\item[---] we also require the following technical property: whenever for $f\colon X\to Y$, $\varphi\colon X\to Q$, $\psi\colon Y\to Q$ we have $\psi(y) \leqslant \bigvee_{\{x\tze fx=y\}} \varphi(x)$, then $\xi(U(\psi)(\sigma))\leqslant \bigvee_{\{\nu\tze U(f)(\nu)=\sigma\}} \xi(U(\varphi)(\nu))$ holds.
\end{itemize}

In conclusion, the triple $(\Ultra, \V, \xi)$ is a strict topological theory in the sense of \cite{Hof_TopTh}. In the following subsections we summarise the main aspects of the theory of $\Tth$-categories, referring to \cite{CT_MultiCat,Hof_TopTh,CH_Compl,Hof_Cocompl,CH_CocomplII} for further details. We remark that many notions and results do not differ dramatically from the $\V$-case, with the notable exception of the dual category and, consequently, the Yoneda lemma (see Proposition \ref{DistasFunTV}, Lemma \ref{YonedaTV} and Corollary \ref{YonedaLemmaTV} below). Our main contribution here is the introduction and study of continuity (see Subsection \ref{ContTV}), which has to face yet another problem: the lifting of distributors is not always available in the $\Tth$-case. Therefore we cannot use freely the way-below distributor $\ddn$, however, we prove that local versions of $\ddn$ do exist and can often be used instead.

\begin{remark}
By the Fundamental Theorem of Compact Semilattices (VI-3.4 of \cite{Book_ContLat}), the only compact Hausdorff topology on $\V$ making $\wedge$ continuous is the Lawson topology. Therefore, if the tensor on $\V$ is given by infimum, the Lawson topology is the only compact Hausdorff topology on $\V$ turning $(\Ultra, \V, \xi)$ into a strict topological theory.
\end{remark}

\subsection{$\Tth$-relations}
A $\V$-relation of the form $\alpha\colon UX\relto Y$ we call
$\Tth$-relation from $X$ to $Y$, and write $\alpha\colon X\krelto
Y$. For $\Tth$-relations $\alpha\colon X\krelto Y$ and $\beta\colon
Y\krelto Z$ we define the Kleisli convolution
$\beta\kleisli\alpha\colon X\krelto Z$ as
$\beta\kleisli\alpha=\beta\cdot  U \alpha\cdot m_X^{op}$. Kleisli
convolution is associative and has the $\Tth$-relation
$e_X^{op}\colon X\krelto X$ as a lax identity: $a\kleisli
e_X^{op}=a$ and $e_Y^{op}\kleisli a\geqslant a$ for any $a\colon
X\krelto Y$. We call $a\colon X\krelto Y$ unitary if
$e_Y^{op}\kleisli a= a$. Furthermore, for a $\Tth$-relation
$\alpha\colon X\krelto Y$, the composition function
$(-)\kleisli\alpha$ still has a right adjoint $(-)\whiteleft\alpha$
(we define $\gamma\whiteleft\alpha\colon =\gamma\blackleft(
U (\alpha)\cdot m_X^{op})$) but $\alpha\kleisli (-)$ in general
does not.

\subsection{$(\Ultra,\V)$-categories}
An $(\Ultra,\V)$-category is a pair consisting of a set $X$ and an
$\Tth$-endorelation $X(-,-)\colon X\krelto X$ such that
$e_X^{op}\leqslant X$ and $X\kleisli X\leqslant X$. Expressed
elementwise, these conditions become 
\begin{align*}
 \unit\leqslant X(e_X(x),x) &&\text{and}&&
UX(\Upsilon,\upsilon)\otimes X(\upsilon,x)\leqslant X(m_X(\Upsilon),x)
\end{align*}
for all $\Upsilon\in UUX$, $\upsilon\in UX$ and
$x\in X$. A function $f\colon X\to Y$ between
$(\Ultra,\V)$-categories is a $(\Ultra,\V)$-functor if $f\cdot
X\leqslant Y\cdot Tf$, which in pointwise notation reads as
$X(\upsilon,x)\leqslant Y(Uf(\upsilon),f(x))$ for all $\upsilon\in
UX$, $x\in X$. If we have above even equality, we call $f\colon X\to
Y$ fully faithful. The resulting category of
$(\Ultra,\V)$-categories and $\Tth$-functors we denote as
$\Cat{\Tth}$. The quantale $\V$ becomes a $(\Ultra,\V)$-category
$\V=(\V,\hom_\xi)$, where $\hom_\xi\colon U
\V\times\V\to\V,\;(\sigma,v)\mapsto\hom(\xi(\sigma),v)$. By $|X|$ we
denote the $(\Ultra,\V)$-category $(UX,m_X)$. There is also a free
$(\Ultra,\V)$-category on a set $X$ given by $(X,e_X^{op})$. We have
a canonical forgetful functor
$\ForgetToV\colon\Cat{\Tth}\to\Cat{\V}$ sending a
$(\Ultra,\V)$-category $X$ to its underlying $\V$-category
$\ForgetToV\!X=(X,X\cdot e_X)$. Furthermore, $S$ has a left adjoint
$\ForgetToVAd\colon\Cat{\V}\to\Cat{\Tth}$ defined by
$\ForgetToVAd\!X=(X,e_X^{op}\cdot UX)$, for each $\V$-category
$X$.
\begin{example}\label{ExTopSp}
For $\V=\dwa$, a $\Tth$-category is a topological space presented
via its ultrafilter convergence structure, and a function $f:X\to Y$
between topological spaces is continuous if and only if it is
$\Tth$-functor (see \cite{Bar_RelAlg}). The functor
$\ForgetToV\colon\TOP\to\ORD$ sends a topological space $X$ to the
ordered set $X$ where $x\leqslant y:\iff \dot{x}\to y\iff
y\in\overline{\{x\}}$, and its left adjoint $\ForgetToVAd\colon\ORD\to\TOP$ takes an ordered set to its Alexandroff space. Note that we consider $X$ here with the dual
of the specialization order. The quantale $\dwa$ becomes the Sierpi\'nski space with $\{1\}$ closed, and $|X|$ is the \v{C}ech-Stone compactification of the discrete space $X$.
\end{example}
% \begin{example}
% For $\V=[0,\infty]$, $\Tth$-categories correspond to approach spaces and $\Tth$-functors to non-expansive maps (see \cite{Low_ApBook}).
% \end{example}

There is yet another functor connecting $(\Ultra,\V)$-categories
with $\V$-categories, namely $\MFunctor\colon\Cat{\Tth}\to\Cat{\V}$
which sends a $(\Ultra,\V)$-category $X$ to the $\V$-category
$(UX,U  X\cdot m_X^{op})$. These functors are all needed to
define the dual of a $(\Ultra,\V)$-category $X$, namely $X^\op:
=\ForgetToVAd((\MFunctor\! X)^\op)$.

As studied in \cite{Hof_TopTh} the tensor product of $\V$ can be transported to $\Cat{\Tth}$ by putting $X\otimes Y := X\times Y$ with structure $(X\otimes Y)(\sigma,(x,y))=X(\upsilon,x)\otimes Y(\nu,y)$, where $\sigma\in U(X\times Y)$, $x\in X$, $y\in Y$, $\upsilon=U\pi_1(\sigma)$ and $\nu=U\pi_2(\sigma)$. The $(\Ultra,\V)$-category $E=(1,\unit)$ is a $\otimes$-neutral object, where $1$ is a singleton set and $\unit\colon U1\times 1\to\V$ the constant relation with value $\unit\in\V$. In general, this constructions does not result in a closed structure on $\Cat{\Tth}$; however we have that: $|X|\otimes(-)\colon \Cat{\Tth}\to\Cat{\Tth}$ has a right adjoint $(-)^{|X|}\colon \Cat{\Tth}\to\Cat{\Tth}$.
% The structure on $\V^{|X|}$ is given by the formula
% $$\V^{|X|}(\rho,\psi) = \bigwedge_{\{\varrho\in T(|X|\times\V^{|X|})\tze \varrho\mapsto\rho\}}\V(\xi\cdot U\ev(\varrho),\psi(m_X\cdot U\pi_1(\varrho))),$$ for each $\rho\in U\V^{|X|}$ and $\psi\in\V^{|X|}$. In  particular for $\rho=e_{\V^{|X|}}(\varphi)$ we have:
% $$\V^{|X|}(e_{\V^{|X|}}(\varphi),\psi) =\bigwedge_{\upsilon\in UX}\V(\varphi\upsilon,\psi\upsilon).$$

\subsection{$\Tth$-distributors}
Let $X$ and $Y$ be $\Tth$-categories and $\varphi\colon X\krelto Y$
be a $\Tth$-relation. We call $\varphi$ a $\Tth$-distributor, and
write $\varphi\colon X\kmodto Y$, if $\varphi\kleisli X=\varphi$ and
$Y\kleisli \varphi=\varphi$. Kleisli convolution is associative, and
it follows that $\psi\kleisli\varphi$ is a $\Tth$-distributor if
$\psi\colon Y\kmodto Z$ and $\varphi\colon X\kmodto Y$ are so.
Furthermore, we have $X(-,-)\colon X\kmodto X$ for each
$\Tth$-category $X$, and, by definition, $X(-,-)$ is the identity
$\Tth$-distributor on $X$ for the Kleisli convolution. In other
words, $\Tth$-categories and $\Tth$-distributors form a category,
denoted as $\Mod{\Tth}$, with Kleisli convolution as compositional
structure. In fact, $\Mod{\Tth}$ is an ordered category with the
structure on $\hom$-sets inherited from $\Mat{\Tth}$. Finally, a
$\Tth$-relation $\varphi\colon X\krelto Y$ is unitary precisely if
$\varphi$ is a $\Tth$-distributor $\varphi\colon
(X,e_X^{op})\kmodto(Y,e_Y^{op})$ between the corresponding discrete
$\Tth$-categories.

A $\Tth$-functor $f\colon X\to Y$ induces $\Tth$-distributors $f_*\colon X\krelto Y$ and $f^*\colon
Y\krelto X$ by putting $f_*=Y\cdot Uf$ and $f^*=f^{op}\cdot Y$
respectively. Hence, for $\sigma\in UX$, $\nu\in UY$, $x\in X$ and
$y\in Y$, we have $f_*(\sigma,y)=b(Uf(\sigma),y)$ and
$f^*(\nu,x)=b(\nu,f(x))$. These $\Tth$-distributors form an adjunction $f_*\dashv f^*$ in $\Mod{\Tth}$. Moreover, given a $(\Ultra,\V)$-functor $g\colon Y\to Z$, $g_*\kleisli f_* =(g\cdot f)_*$
and $f^*\kleisli g^*= (g\cdot f)^*$, plus $(1_X)_*=(1_X)^*=X$.

We will often need the following crucial property.
\begin{proposition}\label{DistasFunTV}
For an $\Tth$-relation $\psi\colon X\krelto Y$, the following are equivalent:
\begin{enumerate}
\item[\rm (i)] $\psi\colon X\kmodto Y$ is a $\Tth$-distributor;
\item[\rm (ii)] Both $\psi\colon |X|\otimes Y\to\V$ and $\psi\colon X^\op\otimes Y\to\V$ are $(\Ultra,\V)$-functors.
\end{enumerate}
\end{proposition}
Therefore, each $\Tth$-distributor $\varphi\colon X\kmodto Y$
defines a $(\Ultra,\V)$-functor $\mate{\varphi}\colon Y\to\V^{|X|}$
which factors through the embedding $\widehat{X}\hrw\V^{|X|}$, where
$\widehat{X}=\{\psi\in\V^{|X|}\mid \psi\colon X\kmodto 1\}$ and $1$
denotes the $(\Ultra,\V)$-category $(1,e^{{op}}_1)$.
\[
\xymatrix{Y\ar[r]^{\mate{\varphi}}\ar[dr]_ {\mate{\varphi}}& \V^{|X|}\\ & \widehat{X}\ar@{_(->}[u]}
\]
In particular, for each $\Tth$-category $X$ we have $X(-,-)\colon X\kmodto X$, and therefore obtain the Yoneda $(\Ultra,\V)$-functor
$\yoneda=\mate{X}\colon X\to\widehat{X}$.
The following result is crucial to transport $\V$-categorical ideas into the $\Tth$-setting.
\begin{lemma}\label{YonedaTV}
Let $\psi\colon X\kmodto Z$ and $\varphi\colon X\kmodto Y$ be $\Tth$-distributors. Then, for all $\zeta\in UZ$ and $y\in Y$, $\V^{|X|}(U\mate{\psi}(\zeta),\mate{\varphi}(y)) =(\varphi\homkleisliright\psi)(\zeta,y)$.
\end{lemma}
\begin{corollary}\label{YonedaLemmaTV}
For each $\varphi\in\widehat{X}$ and each $\sigma\in UX$, $\varphi(\sigma)=\V^{|X|}(U\yoneda(\sigma), \varphi)$, that is, $(\yoneda)_*\colon X\kmodto\widehat{X}$ is given by the evaluation map $\ev\colon UX\times\widehat{X}\to\V$. As a consequence, $\yoneda\colon X\to\widehat{X}$ is fully faithful.
\end{corollary}
\begin{example}
We consider the quantale $\V=\dwa$. In Example \ref{ExTopSp} we have already seen that this case captures precisely topological spaces and continuous maps. It is shown in \cite{HT_LCls} that a distributor $X\kmodto 1$ corresponds to a (possibly improper) filter on the lattice of open subsets of $X$, and the ``presheaf space'' $\widehat{X}$ is homeomorphic to the space $F_0(X)$ of all such filters, where the sets
\begin{align*}
\{F\in F_0(X)\mid A\in F\}&& \text{($A\subseteq X$ open)}
\end{align*}
form a basis for the topology on $F_0(X)$ (see also
\cite{Esc_InjSp}). Note that $F\leqslant G$ if and only if
$F\supseteq G$ in the underlying ordered set $\ForgetToV(F_0(X))$.
The Yoneda embedding $\yoneda:X\to F_0(X)$ sends each point $x$ to
the filter $\mathcal{N}(x)$ of all open neighbourhoods of $x$.
\end{example}

\subsection{$J$-cocomplete $\Tth$-categories}

As in the case of $\V$-categories, we consider cocompleteness and continuity with respect to chosen distributors. To do so, let $\Mod{J}$ be a subcategory of $\Mod{(\Ultra,\V)}$ such that, for every $(\Ultra,\V)$-functor $f$, $f^*\in J$ and,
for all $\varphi\colon X\kmodto Y\in\Mod{(\Ultra,\V)}$,
\[
(\forall y\in Y\,.\,y^*\kleisli \varphi\in J)\Rw \varphi\in J.
\]
We write $J(X)$ for the full subcategory of $\hat{X}$ defined by all $J$-distributors of type $X\kmodto 1$. A $(\Ultra,\V)$-category $X$ is $J$-cocomplete if $\yoneda:X\to J(X)$ has a left adjoint $\Sup:J(X)\to X$ in the ordered category $\Cat{(\Ultra,\V)}$. By definition, $\Sup:J(X)\to X$ is a $(\Ultra,\V)$-functor such that, for all $x\in X$ and $\Upsilon\in U(JX)$,
\[
 X(U\Sup(\Upsilon),x)=\hat{X}(\Upsilon,\yoneda(x)).
\]
It is worthwhile to mention that any left inverse $(\Ultra,\V)$-functor $\Sup:J(X)\to X$ of $\yoneda_X$ is actually a left adjoint. However, we should also mention the situation slightly differs here from the $\V$-case. As before, the map $\Sup$ gives for each $\psi\in J(X)$ a supremum, i.e.\ $x\in X$ with $x_*=1_X\whiteleft \psi$. But it is \emph{not true} that $X$ is $J$-cocomplete if each $\psi\in J(X)$ has a supremum $x$ in $X$ since the induced map $\Sup:J(X)\to X,\psi\mapsto x$ is in general only a $\V$-functor.
\begin{example}
Let $X$ be a complete ordered set. We define a sub-basis $\mathcal{B}$ for a topology on $X$ as follows: $A\in\mathcal{B}$ whenever $A$ is down-closed and, for any $B\subseteq X$, $\bigwedge B\in A$ implies $B\cap A\neq\varnothing$. One easily verifies that the underlying order of the induced topology is just the order we started with, moreover, $X$ is the only neighbourhood of the top-element of $X$. Hence, each filter $\psi$ (of opens) converges and has indeed a smallest convergence point. To see this, let $B$ be the set of all convergence points of $\psi$, and put $y=\bigwedge B$. Let $A\in\mathcal{B}$ with $y\in A$. Then there is some $x\in B\cap A$ and $A\in\psi$ since $\psi$ converges to $x$. Consequently, $y$ is the smallest convergence point of $\psi$. Therefore each distributor $\psi:X\kmodto 1$ has a supremum in $X$ but $X$ cannot be an injective space if the dual of $X$ is not a continuous lattice.
\end{example}
Hence, a $(\Ultra,\V)$-category $X$ is $J$-cocomplete if and only if each $\psi:X\kmodto 1$ in $\Mod{J}$ has ``continuously'' a supremum. We remark \emph{en passant} that, if one allows distributors in $\Mod{J}$ with arbitrary codomain, then again one has that $X$ is $J$-cocomplete if and only if each $\psi:X\kmodto Y$ in $\Mod{J}$ has a supremum in $X$ (see \cite{Hof_Cocompl,CH_CocomplII}). This is one of the reasons why we prefer to define relative cocompletness with respect to a category $\Mod{J}$ of distributors rather then a choice of presheafs $X\kmodto 1$, for each $X$.

\subsection{$J$-continuous $\Tth$-categories}\label{ContTV}

We come now to our main purpose in this section and introduce $J$-continuous $\Tth$-categories. Due to the difficulties described in the previous subsection, we cannot introduce $J_S(X)$ as in Section \ref{section:JcocoVcat} and therefore define $J$-continuity only for $J$-cocomplete $\Tth$-categories.
\begin{definition}
A $J$-cocomplete $\Tth$-category $X$ is called \emph{$J$-continuous} if the
$\Tth$-functor $\Sup:JX\to X$ has a left adjoint in $\Cat{\Tth}$.
\end{definition}
As in the $\V$-case, such a left adjoint $\Tth$-functor $X\to JX$
corresponds to a $\Tth$-distributor $\ddn:X\kmodto X$ which
necessarily belongs to $J$ and, moreover, must be the lifting $\ddn
= \Sup^* \whiteright \yoneda_*$ of $\yoneda_*:X\modto JX$ along
$\Sup^*:X\modto JX$. However, an immediate problem in generalizing
the way-below relation to an $\Tth$-distributor in an analogous way
to the $\V$-distributor case stems from the fact that in general the
lifting $\whiteright$ between $\Tth$-distributors does not exist. We
deal first with this problem.

\begin{lemma}
Let $\psi\colon UY\relto X$ and $\varphi\colon Z\relto X$ be $\V$-relations, and let  $\varphi\homcompleft\psi\colon UY\relto Z$ the lifting of $\psi$ along $\varphi$ in $\Mat{\V}$.
\[
\xymatrix{X & UY\ar|-{\object@{|}}[l]_{\psi}\ar@{.{>}}|-{\object@{|}}[dl]^{\varphi \homcompleft \psi}\\
  Z\ar|-{\object@{|}}[u]^\varphi}\]
If $\psi$ is a unitary $\Tth$-relation $\psi\colon Y\krelto X$, then so is $\varphi\homcompleft\psi\colon Y\krelto Z$.
\proof
We have to show that
\[
 \varphi\homcompleft\psi\colon |Y|\otimes Z_D\to\V,\;(\nu,z)=\bigwedge_{x\in X}\V(\varphi(z,x),\psi(\nu,x))
\]
is a $(\Ultra,\V)$-functor, where $Z_D$ denotes the free
$\Tth$-category $Z_D=(Z,e_Z^{op})$ on the set $Z$. Since
$\bigwedge:\V^{X_D}\to\V$ is a $(\Ultra,\V)$-functor, it is enough
to show functoriality of
\[
 |Y|\otimes Z_D\otimes X_D\to\V,\;(\nu,z,x)=\V(\varphi(z,x),\psi(\nu,x)).
\]
But this function can be expressed as a composite of $(\Ultra,\V)$-functors
\[
|Y|\otimes Z_D\otimes X_D\to Z_D\otimes X_D\otimes |Y|\otimes X_D\xrightarrow{\;\varphi\otimes\psi\;}
\V_D\otimes\V\to\V.
\]
Note that we use here symmetry of the tensor product $\otimes$ and functoriality of $\Delta_X\colon X_D\to X_D\otimes X_D$.\qed
\end{lemma}
\begin{lemma}
Let $\varphi\colon X\krelto Y$ and $\psi\colon Y\krelto Z$ be $\Tth$-relations. Furthermore, assume that $\varphi$ is unitary and $Y$ finite. Then $\psi\kleisli\varphi=\psi\cdot e_Y\cdot\varphi$.
\begin{proof}
Just observe that
\[
\psi\kleisli\varphi=\psi\cdot U\varphi\cdot m_X^{op}=\psi\cdot
e_Y\cdot e_Y^{op}\cdot U\varphi\cdot m_X^{op}=\psi\cdot
e_Y\cdot\varphi.\qedhere
\]
\end{proof}
\end{lemma}
\begin{lemma}
\label{lemma:Ulifting}
For all $\Tth$-distributors $\varphi\colon Y\kmodto X$ and $\psi\colon 1\kmodto X$, $\varphi$ has a lifting along an $\psi$ in $\Mod{(\Ultra,\V)}$ which is given by $\psi\whiteright \varphi = \psi\cdot e_1\blackright \varphi$.
\[
\xymatrix{X & Y\ar@{-_{>}}|-{\object@{o}}[l]_{\varphi}\ar@{.{>}}|-{\object@{o}}[dl]^{\psi\whiteright \varphi}\\
  1\ar@{-_{>}}|-{\object@{o}}[u]^\psi}
\]

\proof Let $\gamma\colon Y\krelto 1$ be an unitary $\Tth$-relation. Then $\psi\kleisli \gamma\leqslant \varphi$ if and only if $\psi\cdot e_1\cdot\gamma\leqslant \varphi$ if and only if $\gamma\leqslant \psi\cdot e_1 \blackright \varphi$. \qed
\end{lemma}

By analogy with $\V$-distributors, define $v\colon X\kmodto X$ to be
\begin{itemize}
\item \emph{auxiliary}, if $v\leqslant X$;
\item \emph{approximating}, if: $v\in J$, and $X\whiteleft v = X$;
\item \emph{interpolative}, if $v\leqslant v\kleisli v$.
\end{itemize}
We call a $\Tth$-distributor $v:X\kmodto Y$
\begin{itemize}
\item \emph{$J$-cocontinuous} if $\Sup^*\kleisli v = \yoneda_*\kleisli v$.
\end{itemize}
Any approximating $\Tth$-distributor is auxiliary, and any approximating $J$-cocontinuous $\Tth$-distributor is interpolative. Furthermore, the composition of approximating $\Tth$-distributors is again approximating (compare with Lemmata \ref{lem:approxV} and \ref{lem:approxScottV}).

With the same proof as for Proposition \ref{ScottDistFunV} one verifies that $v\colon X\kmodto Y$ is $J$-cocontinuous if and only if the $\Tth$-functor $\mate{v}\colon Y\to\widehat{X}$ is $J$-cocontinuous.

We also define the way-below $\Tth$-distributor $\ddn\colon X\kmodto
X$ as the lifting of $\yoneda_*\colon X\kmodto JX$ along
$\Sup^*\colon X\kmodto JX$, whenever it exists. Since we do not
have in general the way-below distributor `globally', we define its
`local' version at $x\in X$ to be the lifting of $\yoneda_*$ along
$\Sup^*\kleisli x_*$,
\[\xymatrix{JX & X\ar@{-_{>}}|-{\object@{o}}[l]_{\yoneda_*}\ar@{._{>}}|-{\object@{o}}[dl]^{\ \ \Downarrow_x := (\Sup^*\kleisli x_*) \whiteright \yoneda_*}\\
  1\ar@{-_{>}}|-{\object@{o}}[u]^{\Sup^*\kleisli x_*}}\]
which does exist for each $\Tth$-category $X$ and each $x\in X$. Of course, if $\ddn$ exists on $X$, then $\ddn_x=x^*\kleisli\ddn$ for each $x\in X$.
\begin{lemma}
For every $\Tth$-category $X$, the map $\ddn_{-}\colon X\to\widehat{X},\,x\mapsto\ddn_x$ is a $\V$-functor.
\end{lemma}
\begin{proof}
For any $x,y\in X$, we have to show that
\[
 X(x,y)\le\ddn_y\whiteleft \ddn_x.
\]
First note that $X(x,y)=y^*\kleisli x_*$. Now,
\[
y^*\kleisli x_*\leqslant (\Sup^*\kleisli
y_*\whiteright\yoneda_*)\whiteleft (\Sup^*\kleisli
x_*\whiteright\yoneda_*)
\]
if and only if
\[
y^*\kleisli x_*\kleisli (\Sup^*\kleisli
x_*\whiteright\yoneda_*)\leqslant \Sup^*\kleisli
y_*\whiteright\yoneda_*,
\]
which in turn is equivalent to
\[
\Sup^*\kleisli y_*\kleisli y^*\kleisli x_*\kleisli (\Sup^*\kleisli
x_*\whiteright\yoneda_*)\leqslant \yoneda_*;
\]
and this is indeed true since $y_*\kleisli y^*\leqslant X$.
\end{proof}
\noindent So far we are not able to prove or disprove that $\ddn_{-}$ is a $\Tth$-functor. Of course, $\ddn_{-}$ is a $\Tth$-functor if $X$ is $J$-continuous, since in this case $\ddn_{-}=\mate{\ddn}$.
\begin{proposition}
A $J$-cocomplete $\Tth$-category $X$ is $J$-continuous if and only if $\ddn_{-}$ is a $\Tth$-functor and, for each $x\in X$, $\ddn_x\in J(X)$ and $X\whiteleft \ddn_x = x_*$.
\end{proposition}
\begin{proof}
Clearly, the conditions are necessary. Assume now that $\ddn_{-}$ is
a $\Tth$-functor and $\ddn_x\in J(X)$ and $X\whiteleft \ddn_x = x_*$
for each $x\in X$. Hence $\ddn_{-}$ is of type $X\to J_SX$ and we
have $\Sup\cdot\ddn_{x}\cong x$. Let now $\psi\in J_SX$. Then
\[
\ddn_{\Sup\psi}(\nu) = \bigwedge_{\varphi\in J_SX}
Q(X(\dot{\Sup\psi},\Sup\varphi), \varphi(\nu)) \le
Q(X(\dot{\Sup\psi},\Sup\psi), \psi(\nu))\le\psi(\nu),
\]
hence $\ddn_{-}\cdot\Sup\leqslant 1_{J_SX}$, and therefore
$\ddn_{-}\dashv\Sup$.
\end{proof}

\noindent In general, for a distributor $v:X\kmodto X$ and $x\in X$,
we consider its local version $v_x:X\kmodto 1$ at $x$ defined as
$v_x := x^*\kleisli v$. Observe that for any two $\Tth$-distributors
$v,w$ of type $X\kmodto X$, if $v_x\leqslant w_x$ for all $x\in X$,
then $v\leqslant w$, since $v_x\leqslant w_x$ iff $v(\nu,x)\leqslant
w(\nu, x)$ for all $\nu$, iff $v\leqslant w$. Furthermore, we call a
$\Tth$-distributor $v\colon X\kmodto X$ with $v\in J$
\emph{approximating at $x\in X$} if $v_x$ satisfies  $X\whiteleft
v_x = x_*$. The counterparts to Lemma \ref{charapproxV} and
Lemma \ref{lemma:AppImpliesLeqWaybelow} read as follows.
\begin{proposition}
A $\Tth$-distributor $v\colon X\kmodto X$ is approximating at $x$,
for every $x\in X$, if and only if its mate $\mate{v}$ is of type
$X\to JX$ and $\Sup\cdot\mate{v}\cong 1_X$.
\end{proposition}

\begin{lemma}
If $v\colon X\kmodto X$ is approximating at $x\in X$, then
$\ddn_x\leqslant v_x$. \proof $\ddn_x(\nu) = \bigwedge_{\varphi\in
J_SX} Q(X(\dot{x},\Sup\varphi), \varphi(\nu)) \leqslant
Q(X(\dot{x},\Sup v_x), v_x(\nu)) = v_x(\nu)$.\qed
\end{lemma}

Hence, if the way-below distributor exists, it is smaller then any approximation distributor. In particular, $\ddn$ is necessarily auxiliary. As in the $\V$-case we deduce:
\begin{corollary}
If $\ddn$ exists and is approximation, then $\ddn$ is interpolative.
\end{corollary}
\begin{lemma}
Let $v\colon X\kmodto X$ be auxiliary and $J$-cocontinuous. Then, for each $x\in X$, $v_x\le\ddn_x$. Hence, if the way-below distributor $\ddn$ exists, then $v\le\ddn$.
\end{lemma}
\begin{lemma}
Let $v\colon X\kmodto X$ be interpolative such that $\Sup^*\kleisli v\le\yoneda_*$. Then $v$ is $J$-cocontinuous.
\end{lemma}
Of course, $\Sup^*\kleisli v\le\yoneda_*$ is equivalent to $v\le\ddn$ assuming that the way-below distributor $\ddn$ exists.
\begin{lemma}
Let $\alpha\colon X\to JX$ be a $J$-cocontinuous $\Tth$-functor
with $\Sup\alpha \cong 1$. Then $\alpha \dashv \Sup$.
\end{lemma}
\begin{theorem}
Let $v\colon X\kmodto X\in J$. Then the following are equivalent:
\begin{enumerate}
\renewcommand{\theenumi}{\roman{enumi}}
\item $\mate{v}$ is of type $X\to JX$ and $\mate{v}\dashv\Sup$,
\item $v$ is approximating and provides the lifting of $S^*$ along $\yoneda_*$, i.e.\ $v=\ddn$,
\item $v$ is approximating and $J$-cocontinuous,
\item $v$ is approximating at $x\in X$ for every $x\in X$ and $J$-cocontinuous,
\item $v$ is approximating at $x\in X$ for every $x\in X$ and $\mate{v}\colon X\to JX$ is $J$-cocontinuous,
\item for all $\sigma\in UX$ and $\psi\in JX$ we have $\widehat{X}(U\mate{v}(\sigma),\psi) = X(\sigma,\Sup(\psi))$.
\end{enumerate}
\end{theorem}

\begin{theorem}
The following are equivalent, for a $J$-cocomplete $\Tth$-category $X$.
\begin{enumerate}
\renewcommand{\theenumi}{\roman{enumi}}
\item $X$ is $J$-continuous,
\item The way-below $\Tth$-distributor $\ddn \colon X\modto X$ exists and is approximating,
\item There exists a $J$-cocontinuous approximating $\Tth$-distributor $v \colon X\modto X$,
\item There exists a $J$-cocontinuous $\Tth$-distributor $v \colon X\modto X$ which is approximating at $x$, for each $x\in X$.
\end{enumerate}
\end{theorem}

\begin{example}
We consider the quantale $\V=\dwa$, that this, topological spaces and continuous maps. We start with the absolute case $J=\Mod{\Tth}$. A topological space $X$ is cocomplete if and only if $\yoneda:X\to F_0(X)$ has a left adjoint $\Sup:F_0(X)\to X$ in $\TOP$, which is equivalent to $\ForgetToV(X)^\op$ being a continuous lattice. Here $\Sup$ associates to each filter $F\in F_0(X)$ its smallest convergence point with respect to the order in $\ForgetToV(X)$. Furthermore, the local version $\ddn_x$ of the way-below distributor is given by the filter
\[
\ddn_x=\left\langle\bigcup\{F\in F_0(X)\mid x\le\Sup(F)\}\right\rangle
\]
generated by $\bigcup\{F\in F_0(X)\mid x\le\Sup(F)\}$. A space $X$ is $J$-continuous if and only if $\ddn_{-}:X\to F_0(X)$ is continuous and every $x\in X$ is the smallest convergence point of $\ddn_x$. If $X$ is cocomplete, then continuity of $\ddn_{-}:X\to F_0(X)$ reduces to Scott-continuity of the monotone map $\ddn_{-}:\ForgetToV(X)^\op\to (F_0(X),\subseteq)$ in the usual order-theoretic sense. So far we are not able to give a more elementary topological description of (absolute) continuity in topological spaces, however, we remark that

\begin{itemize}
\item each space of the form $F_0(X)$ is cocomplete and $J$-continuous, and more general, a topological T$_0$ space $X$ is continuous if and only if it is the filter space of a frame (this will be the topic of a forthcoming paper).
\item and therefore every T$_0$-space can be embedded into a cocomplete and continuous space.
\end{itemize}

\smallskip\noindent
We finish this paper by mentioning two more examples.

\smallskip\noindent
For $J$ being the class of all right adjoint distributors, a topological space $X$ is $J$-cocomplete if and only if it is weakly sober \cite{CH_Compl}, and every topological space is $J$-continuous.

\smallskip\noindent
Further possible choices of $J$ are discussed in \cite{CH_CocomplII}. For instance, we may consider the class $J$ of all those $\Tth$-distributors $\varphi\colon X\kmodto Y$ for which $\varphi\kleisli (-):\mathsf{Dist}(1,X)\to\mathsf{Dist}(1,Y)$ preserves certain infima. Note that a distributor $1\kmodto X$ corresponds to a continuous map $X\to\dwa$, which in turn corresponds to a closed subset of $X$. Hence $\mathsf{Dist}(1,X)$ is isomorphic to the lattice of closed subsets of $X$. In particular, we can chose $J=\{\varphi\colon X\kmodto Y\mid \varphi\kleisli (-)\text{ preserves the top element}\}$. Then
\[
 \varphi\in J\iff \forall y\in Y\exists \nu\in UX\,.\,\nu \varphi y.
\]
Hence, a distributor $\varphi\colon X\kmodto 1$ belongs to $J$ if and only if it corresponds to a proper filter. Therefore
\[
\ddn_x=\left\langle\bigcup\{F\in F_0(X)\mid x\le\Sup(F)\text{ and $F$ is proper}\}\right\rangle,
\]
and a continuous map $f:X\to Y$ is $J$-dense precisely if it is dense in the usual topological sense. Consequently, $X$ is $J$-cocomplete if and only if $X$ is densely injective. Finally, $X$ is $J$-continuous if and only if $\ddn_{-}:X\to F_0(X)$ is continuous and, for every $x\in X$, the filter $\ddn_x$ is proper and $x$ is its smallest convergence point.
\end{example}

\bibliographystyle{halpha}

\begin{thebibliography}{vBMOW05}

\bibitem[ALR03]{ALR03} Adamek,~J., Lawvere,~F.W. and Rosicky,~J. (2003) Continuous categories revisited, \emph{Theory and Applications of Categories} \textbf{11}(11), pp.~252--282.

\bibitem[AJ94]{abramsky94} Abramsky,~S. and Jung,~A. (1994) Domain Theory. In S.Abramsky, D.M. Gabbay and T.S.E. Maibaum, editors, \emph{Handbook of Logic in Computer Science} \textbf{3}, pp.~1--168, Oxford University Press.

\bibitem[AK88]{albert:kelly} Albert,~M.H. and Kelly,~G.M. (1988) The closure of a class of colimits, \emph{Journal of Pure and Applied Algebra} \textbf{51}, pp.~1--17.

\bibitem[AR89]{america89} America, P. and Rutten, J.J.M.M. (1989) Solving reflexive domain equations in a category of complete metric spaces, \emph{Journal of Computer and System Sciences} \textbf{39}(3), pp.~343--375.

\bibitem[deBZ82]{deBZ82} de Bakker, J. W. and Zucker, J. (1982) Processes and the Denotational Semantics of Concurrency, \emph{Information and Control} \textbf{54}, pp.~70--120.

\bibitem[Bar70]{Bar_RelAlg} Barr, M. (1970) Relational algebras, In {\em Reports of the Midwest Category Seminar, IV}, Lecture Notes in Mathematics, Vol. 137. Springer, Berlin, pp.~39--55.

\bibitem[BvBR96]{bonsangue96} Bonsangue, M. M., van Breugel, F. and Rutten, J. J. M. M. (1996) Alexandroff and {S}cott {T}opologies for {G}eneralized {M}etric {S}paces, Papers on General Topology and Applications: Eleventh Summer Conference at University of Southern Maine, S. Andima et al. eds., Annals of the New York Academy of Sciences, pp.~49--68.

\bibitem[vBr01]{vBr01} van Breugel,~F. (2001) An Introduction to Metric Semantics: Operational and Denotational Models for Programming and Specification Languages, \emph{Theoretical Computer Science} \textbf{258}(1-2), pp.~1--98.

\bibitem[vBW05]{vBW05} van Breugel,~F. Worrell,~J (2005) A Behavioural Pseudometric for Probabilistic Transition Systems, \emph{Theoretical Computer Science} \textbf{331}(1), pp.~115--142.

\bibitem[vBMOW05]{vBMOW05} van Breugel,~F., Mislove,~M., Ouaknine,~J. and Worrell,~J. (2005) Domain Theory, Testing and Simulation for Labelled Markov Processes, \emph{Theoretical Computer Science} \textbf{333}(1-2), pp.~171--197.

\bibitem[vBHMW07]{vBHMW07} van Breugel,~F., Hermida,~C., Makkai,~M. and Worrell,~J. (2007) Recursively Defined Metric Spaces without Contraction, \emph{Theoretical Computer Science} \textbf{380}(1-2), pp.~143--163.

\bibitem[BvBR98]{rutten:yoneda} Bonsangue,~M.M., van Breugel,~F. and Rutten,~J.J.M.M. (1998) Generalized Metric Spaces: Completion, Topology, and Powerdomains via the Yoneda Embedding, \emph{Theoretical Computer Science} \textbf{193}(1-2), pp.~1--51.

\bibitem[CH03]{CH_TopFeat}
Clementino, M.M. and Hofmann, D. (2003) Topological features of lax
  algebras, \emph{Applied Categorical Structures} \textbf{11}~(3), pp.~267--286.

\bibitem[CH09a]{CH_Compl}
Clementino, M.M. and Hofmann, D. (2009) {L}awvere completeness in {T}opology, {\em Applied Categorical Structures} \textbf{17}, pp.~175--210; \url{arXiv:math.CT/0704.3976}.

\bibitem[CH09b]{CH_CocomplII}
Clementino, M.M. and Hofmann, D. (2009) Relative injectivity as cocompleteness for a class of distributors, {\em Theory and Applications of Categories} \textbf{21}, pp.~210--230.

\bibitem[CHT04]{CHT_OneSetting}
Clementino, M.M., Hofmann, D. and Tholen, W. (2004) One setting for all: metric, topology, uniformity, approach structure, {\em Applied Categorical Structures}, \textbf{2}, pp.~127--154.

\bibitem[CT03]{CT_MultiCat} Clementino, M.M and Tholen, W. (2003) Metric, topology and multicategory --- a common approach, \emph{Journal of Pure and Applied Algebra} \textbf{179}(1-2), pp.~13--47.

\bibitem[Die99]{Diers_AffAlgSetTh} Diers , Y. (1999) Affine algebraic sets relative to an algebraic theory, \emph{Journal of Geometry} \textbf{65}(1-2), pp.~54--76.

\bibitem[EK66]{eilenberg66} Eilenberg,~S. and Kelly,~G. M (1966) Closed Categories, \emph{Proceedings of the Conference on Categorical Algebra}, La Jolla 1965, S. Eilenberg, D. K. Harrison, S. MacLane and H. R{\"o}hrl eds., Springer Verlag, pp.~421--562.

\bibitem[Esc97]{Esc_InjSp}
Escard{\'o},~M. (1997) Injective spaces via the filter monad, In {\em Proceedings of the 12th Summer Conference on General Topology and its Applications (North Bay, ON, 1997)}, volume~22, pp.~97--100.

\bibitem[Fl97]{continuity} Flagg,~R.C. (1997) Quantales and continuity spaces, \emph{Algebra Universalis} \textbf{37}, pp.~257--276.

\bibitem[FK95]{flagg:fixed} Flagg,~R.C., Kopperman,~R. (1995) Fixed points and reflexive domain equations in categories of continuity spaces. \emph{Electronic Notes in Theoretical Computer Science} \textbf{1}.

\bibitem[FK97]{flagg:quantales} Flagg,~R.C., Kopperman,~R. (1997) Continuity Spaces: Reconciling Domains and Metric Spaces, \emph{Theoretical Computer Science} \textbf{177}(1), pp.~111-138.

\bibitem[FS02]{flagg:essence} Flagg,~R.C., S\"{u}nderhauf,~P. (2002) The essence of ideal completion in quantitative form. \emph{Theoretical Computer Science} \textbf{278}(1-2), pp.~141--158.

\bibitem[FSW96]{FSW} Flagg,~R.C., S\"{u}nderhauf,~P. and Wagner,~K.R. (1996) A Logical Approach to Quantitative Domain Theory, \emph{Topology Atlas Preprint no. 23}, available on-line at: \url{http://at.yorku.ca/e/a/p/p/23.htm}.

\bibitem[GHK+03]{Book_ContLat} Gierz,~G., Hofmann,~K.H., Keimel,~K., Lawson,~J.D., Mislove,~M. and Scott,~D.S. (2003) Continuous lattices and domains. Volume 93 of \emph{Encyclopedia of mathematics and its applications}.

\bibitem[Hof07]{Hof_TopTh} Hofmann,~D. (2007) Topological theories and closed objects, \emph{Advances in Mathematics} \textbf{215}, pp.~789--824.

\bibitem[Hof10]{Hof_Cocompl} Hofmann, D. (2010) Injective spaces via adjunction. \emph{Accepted for publication in Journal of Pure and Applied Algebra}, \url{arXiv:math.CT/0804.0326}.

\bibitem[HT08]{HT_LCls} Hofmann,~D. and Tholen,~W. (2008) {L}awvere completion and separation via closure, {\em Accepted for publication in Applied Categorical Structures}, \url{arXiv:math.CT/0801.0199}.

\bibitem[Joh82]{joh82} Johnstone,~P. and Joyal,~A. (1982) Continuous categories and exponentiable toposes, \emph{Journal of Pure and Applied Algebra} \textbf{25}, pp.~255--296.

\bibitem[Kel82]{kelly} Kelly,~G.M. (1982) Basic concepts of enriched category theory, London Mathematical Society Lecture Note Series 64, Cambridge University Press. Also: Reprints in Theory and Applications of Categories, No.~10 (2005).

\bibitem[KS05]{kelly:schmitt} Kelly,~G.M. and Schmitt,~V. (2005) Notes on categories with colimits of some class, \emph{Theory and Applications of categories} \textbf{14}(17), pp.~399--423.

\bibitem[Kop88]{kopperman88} Kopperman,~R. (1988) All topologies come from generalized metrics, \emph{American Mathematical Monthly} \textbf{95}(2), pp.~89--97.

\bibitem[Kos86]{kos86} Koslowski,~J. (1986) Continuous categories, \emph{Lecture Notes in Computer Science} \textbf{239}, pp.~149--161.

\bibitem[KS02]{kunzischellekens} K\"unzi,~Hans-Peter A., Schellekens,~M.P. (2002) On the Yoneda completion of a quasi-metric space, \emph{Theoretical Computer Science} \textbf{278}(1-2), pp.~159-194.

\bibitem[K\"un93]{kunzi93} K{\"u}nzi,~H-P. (1993) {N}onsymmetic topology, Bolyai Society in Mathematical Studies 4, Szeksz\'{a}rd, Hungary (Budapest 1995) pp.~303--338.

\bibitem[LZ07]{zhang07} Lai,~H. and Zhang,~D. (2007) Complete and directed complete Omega-categories.  \emph{Theoretical Computer Science} \textbf{388}(1-3), pp.~1--25.

\bibitem[Law73]{Law_MetLogClo} Lawvere,~F.W. (1973) Metric spaces, generalized logic, and closed categories. \emph{Rend. Sem. Mat. Fis. Milano} \textbf{43} pp. 135--166. Also in: Repr. Theory Appl. Categ. \textbf{1}:1--37, 2002.

\bibitem[Low97]{Low_ApBook} Lowen,~R. (1997) Approach spaces, Oxford Mathematical Monographs, The Clarendon Press Oxford University Press, New York.

\bibitem[Rut96]{rutten96} Rutten,~J.J.M.M. (1996) Elements of generalized ultrametric domain theory, \emph{Theoretical Computer Science}  \textbf{170}, pp.~349--381.

\bibitem[Rut98]{rutten98} Rutten,~J.J.M.M. (1998) Weighted colimits and formal balls in generalized metric spaces, \emph{Topology and its Applications} \textbf{89}, pp.~179--202.

\bibitem[Sch06]{schmitt} Schmitt,~V. (2006) Flatness, preorders and general metric spaces, \emph{to appear in Georgian Mathematical Journal}; \url{arXiv:math.CT/0602463}.

\bibitem[Smy88]{smyth88} Smyth,~M.B. (1988) Quasi-uniformities: Reconciling Domains and Metric Spaces. \emph{Lecture Notes in Computer Science} \textbf{298}, pp.~236--253.

\bibitem[Smy91]{smyth91} Smyth,~M.B. (1991) Totally bounded spaces and compact ordered spaces as domains of computation. In G.M.~Reed, A. W. Roscoe, and R. F. Wachter, editors, Topology and Category Theory in Computer Science, pp.~207--229. Clarendon Press.

\bibitem[Smy94]{smyth94} Smyth,~M.B. (1994) Completeness of quasi-uniform and syntopological spaces. \emph{Journal of the London Mathematical Society} \textbf{49}, pp.~385--400.

\bibitem[SP82]{SP82} Smyth,~M.B. and Plotkin, G. (1982) The category theoretic solution of recursive domain equations, \emph{SIAM Journal of Computing} \textbf{11}(4), pp.~761-783.

\bibitem[Stu07]{Stu_Dynamics} Stubbe, I. (2007) Towards ``dynamic domains'': totally continuous cocomplete {$\mathcal{Q}$}-categories. \emph{Theoretical Computer Science}, \textbf{373}(1-2), pp.~142--160.

\bibitem[S\"un93]{sunderhauf93} S{\"u}nderhauf,~P. The {S}myth-completion of a quasi-uniform space (1993) In: Semantics of Programming Languages and Model Theory, vol. 5 of Algebra, Logic and Applications series, Gordon and Breach Science Publ., M. Droste and Y. Gurevich eds., pp.~189--212.

\bibitem[S\"un95]{sunderhauf95} S{\"u}nderhauf,~P. (1995) Quasi-uniform completeness in terms of {C}auchy nets, \emph{Acta Mathematica Hungarica} \textbf{69}, pp.~47--54.

\bibitem[S\"un97]{sunderhauf97} S{\"u}nderhauf,~P. (1997) {S}myth completeness in terms of nets: the general case, \emph{Quaestiones Mathematicae} \textbf{20}(4), pp.~715--720.

\bibitem[Vic05]{vickers} Vickers,~Steven (2005) Localic Completion of Generalized Metric Spaces {I}, \emph{Theory and Applications of Categories} \textbf{14}, pp.~328--356.

\bibitem[Wag94]{wagner94} Wagner,~K.R. (1994) Solving Recursive Domain Equations with Enriched Categories, PhD Thesis, Carnegie Mellon University.

\bibitem[Wag97]{wagner97} Wagner,~K.R. (1997) Liminf convergence in $\Omega$-categories, \emph{Theoretical Computer Science} \textbf{184}(1-2), pp.~61--104.

\bibitem[ZF05]{zhang} Zhang,~Q.-Y., Fan,~L. (2005) Continuity in quantitative domains \emph{Fuzzy Sets and Systems} \textbf{154}(1), pp.~118--131.

\end{thebibliography}

\end{document}